\documentclass{amsart}
\usepackage[T1]{fontenc}
\usepackage[toc, page]{appendix}
\usepackage{amssymb, amsfonts, latexsym, mathrsfs, amsmath, amsthm, bbm, amstext}
\usepackage{txfonts,bbding,wasysym,pifont,stmaryrd,tipa, bm, amsbsy}
\usepackage[all,cmtip]{xy}
\usepackage{float, mathtools, enumitem}
\usepackage[utf8]{inputenc}
\usepackage{hyperref}
\usepackage{cleveref}
\crefname{section}{§}{§§}
\Crefname{section}{§}{§§}
\DeclarePairedDelimiter\floor{\lfloor}{\rfloor}

\newtheorem{theorem}{Theorem}[section]
\newtheorem{lemma}[theorem]{Lemma}
\newtheorem{proposition}[theorem]{Proposition}
\newtheorem{corollary}[theorem]{Corollary}
\newtheorem{conj}[theorem]{Conjecture}

\theoremstyle{definition}
\newtheorem{definition}[theorem]{Definition}
\newtheorem{example}[theorem]{Example}

\theoremstyle{remark}
\newtheorem{remark}[theorem]{Remark}

\numberwithin{equation}{section}

\DeclareMathAlphabet{\mathpzc}{OT1}{pzc}{m}{it}

\DeclareMathOperator{\sheafHom}{\mathscr{H}\text{\kern -3pt {\calligra\large om}}\,}

\begin{document}

\title{Towards the Generalized Purely Wild Inertia Conjecture for product of Alternating and Symmetric Groups}
\author{Soumyadip Das}
\address{Mathematics,
Chennai Mathematical Institute, H1, SIPCOT IT Park, Siruseri, Kelambakkam 603103, India.}
\email{soumyadipd@cmi.ac.in}



\begin{abstract}
We obtain new evidence for the Purely Wild Inertia Conjecture posed by Abhyankar and for its generalization. We show that this generalized conjecture is true for any product of simple Alternating groups in odd characteristics, and for any product of certain Symmetric or Alternating groups in characteristic two. We also obtain important results towards the realization of the inertia groups which can be applied to more general set up. We further show that the Purely Wild Inertia Conjecture is true for any product of perfect quasi $p$-groups (groups generated by their Sylow $p$-subgroups) if the conjecture is established for individual groups. 
\end{abstract}

\maketitle

\section{Introduction}
The objective of this article is to establish the Generalized Purely Wild Inertia Conjecture for arbitrary finite product of Alternating and Symmetric groups in arbitrary prime characteristic. This is completely achieved in odd characteristics, and with conditions on the degrees of these permutation groups in even characteristic.

Let $p$ be a prime number and $k$ be an algebraically closed filed of characteristic $p$. A finite group $G$ is said to be a \textit{quasi} $p$-\textit{group} if $G$ is generated by all its Sylow $p$-subgroups. In \cite{Abh_57}, Abhyankar posed a conjecture stating that the quasi $p$-groups are precisely the groups occurring as the Galois groups for connected Galois \'{e}tale covers of the affine $k$-line $\mathbb{A}^1$. This is a sharp contrast from the behavior of finite covers defined over $\mathbb{C}$ as there is no non-trivial \'{e}tale cover of the Riemann sphere. The conjecture (known as Abhyankar's Conjecture on the affine line) was proved by Serre (for solvable quasi $p$-groups; \cite{Serre_AC}) and Raynaud (\cite{Raynaud_AC}). It can be seen that any such Galois \'{e}tale cover of $\mathbb{A}^1$ extends uniquely to a Galois cover of the projective $k$-line $\mathbb{P}^1$ that is branched over $\infty$ and is \'{e}tale everywhere else. Moreover, the inertia groups above $\infty$ are conjugate to each other. For a quasi $p$-group $G$, and a subgroup $I \, \subset \, G$, we say that the pair $(G, \, I)$ is \textit{realizable} if there is a connected $G$-Galois cover of $\mathbb{P}^1$ branched only at $\infty$, and at a point above $\infty$, the inertia group is $I$. From the theory of extension of local fields, $I$ is necessarily an extension of a $p$-group by a cyclic group of order prime-to-$p$. The question arises that given a quasi $p$-group $G$ and a subgroup $I \, \subset G$ which is potentially an inertia group (i.e., has the above necessary form), whether the pair $(G, \, I)$ is realizable. Abhyankar conjectured a necessary and sufficient group theoretic condition to this question, now known as the Inertia Conjecture.

\begin{conj}[The IC, {\cite[Section 16]{Abh_01}}]\label{conj_IC}
Let $G$ be a finite quasi $p$-group. Let $I$ be a subgroup of $G$ which is an extension of a $p$-group by a cyclic group of order prime-to $p$. Then the pair $(G, \, I)$ is realizable if and only if the conjugates of $P$ in $G$ generate $G$ (in notation, $G = \, \langle \, P^G \, \rangle$).
\end{conj}

It can be seen that the condition on $I$ in the conjecture is a necessary one. When $I =  P$ is a $p$-group, a special case of the above conjecture is known as the Purely Wild Inertia Conjecture (henceforth referred to as the PWIC).

\begin{conj}[The PWIC, {\cite[Section 16]{Abh_01}}]\label{conj_PWIC}
Let $G$ be a finite quasi $p$-group. Let $P$ be a $p$-subgroup of $G$. Then the pair $(G, \, P)$ is realizable if and only if $G = \langle \, P^G \, \rangle$.
\end{conj}

It is known that the Inertia Conjecture is true when $G$ is a $p$-group. From \cite[Theorem 2]{2}, it follows that for any quasi $p$-group and a Sylow $p$-subgroup $P$ of $G$, the pair $(G, \, P)$ is realizable. In particular, the PWIC is true for any quasi $p$-group whose order is strictly divisible by $p$. In other earlier works, the Inertia Conjecture was shown to be true for a few groups. For a summary on these works and other important developments, see \cite[Section 4]{survey_paper}. Recent works from \cite{DK}, \cite{Das}, \cite{Dean} shed light on several systematic ways to construct the Galois \'{e}tale covers of the affine line. As a consequence, a larger class of groups provide evidence towards the above conjectures. In general, the status of the conjectures remain open at the moment. In \cite{Das}, the PWIC was generalized (the Generalized Purely Wild Inertia Conjecture or GPWIC) to the case with multiple branched points in $\mathbb{P}^1$.

\begin{conj}[The GPWIC, {\cite[Conjecture 6.8]{Das}}]\label{conj_GPWIC}
Let $G$ be a finite quasi $p$-group. Let $P_1, \, \cdots, \,P_r$ be non-trivial $p$-subgroup of $G$ for some $r \geq 1$ such that $G \, = \, \langle P_1^G,\, \cdots, \, P_r^G \rangle$. Let $B = \{x_1, \, \cdots, \, x_r\}$ be a set of closed points in $\mathbb{P}^1$. Then there is a connected $G$-Galois cover of $\mathbb{P}^1$ \'{e}tale away from $B$ such that $P_i$ occurs as an inertia group above the point $x_i$ for $1 \leq i \leq r$.
\end{conj}

When $p$ is an odd prime, $A_d$ is a quasi $p$-group for all $d \geq \text{max}\{5,p\}$; when $p = 2$, $A_d$ for $d \geq 5$ is a quasi $2$-group and $S_d$ for $d \geq 3$ is a quasi $2$-group. Moreover, any product of quasi $p$-groups is again a quasi $p$-group. This way, we have a large class of potential candidates for which the above conjectures can be checked. Our goal in this paper is to establish the above mentioned conjecture for the Alternating and Symmetric groups, and their products. 

When $p$ is an odd prime, \cite[Theorem 1.7]{Das} shows that the GPWIC is true for any product $A_{d_1} \times \cdots \times A_{d_n}$, $n \geq 1$, where each $d_i = p$ or $d_i \geq p+1$ is co-prime to $p$. In Corollary~\ref{cor_GPWIC_Alt_products} we remove this restriction on the $d_i$'s, hence proving the GPWIC for arbitrary product of simple quasi $p$ Alternating groups. In view of the proof of \cite[Theorem 7.4]{Das}, this boils down to showing the following important result.

\begin{theorem}[{Theorem \ref{thm_PWIC_Alternating}}]\label{thm_intro_PWIC_Alternating_multiple_p}
Let $p$ be an odd prime, $r \geq 1$. Then the PWIC (Conjecture~\ref{conj_PWIC}) is true for $A_{rp}$.
\end{theorem}

In fact, we prove the following result which can be applied to a more general set up.

\begin{proposition}[{Proposition \ref{prop_purely_wild_from_tame}}]\label{prop_intro_wild_from_tame}
Let $p$ be a prime. Let $G$ be a finite group, $G_1 \times G_2 \subset G$, where $G_1$ and $G_2$ are quasi $p$-groups. Suppose that for $i = 1, \, 2$, $\lambda_i$ is an element of order $p$ in $G_i$ such that the pair $(G_1 \times G_2, \, \langle \, (\lambda_1, \lambda_2) \, \rangle)$ is realizable. Let $c \in G$ be an element of order prime-to-$p$ that interchanges $\lambda_1$ and $\lambda_2$ via conjugation, i.e. $c^{-1} \, (\lambda_1, \lambda_2) \, c = (\lambda_2, \lambda_1)$. Consider the subgroup $H \coloneqq \langle \, G_1 \times G_2 , \, \langle \, c \, \rangle \, \rangle$ of $G$. Let $T$ be the maximal common quotient of $H$ and $\langle \, c \, \rangle$. Then the pair $(H \times_T \langle \, c \, \rangle, \, \langle \, (\lambda_1, \lambda_2) \, \rangle)$ is realizable.
\end{proposition}

The above result lets us potentially realize new Galois groups starting from realization of a product of groups, while the inertia groups remain the same. The approach is different from the previously employed methods to construct purely wildly ramified covers. Namely, we first enlarge the Galois group, enlarge the inertia by a prime-to-$p$ part and add new branched points. Then we obtain our cover with the desired properties via a suitable pullback under a Kummer cover. Such ideas were implicit in \cite{Das} to obtain Alternating group covers with prescribed inertia groups using Abhyankar-type equations. To prove the proposition, we apply a formal patching result Lemma~\ref{lem_fp_main} of Pries to the given cover and a suitable Harbater-Katz-Gabber cover (see \cite[Theorem~2.3.7]{Pries}; one can `patch' together two covers, one with inertia $\mathbb{Z}/p$ and the other with inertia of the form $\mathbb{Z}/p \rtimes \mathbb{Z}/m$, $(p,m)=1$).

Another application (Corollary~\ref{cor_A_p+1}) of the above mentioned patching result is the Inertia Conjecture (Conjecture~\ref{conj_IC}) for $A_{p+1}$ for any prime $p \geq 5$ (this was proved in \cite[Theorem~5.3]{Das} under the condition that $p \equiv 2 \pmod{3}$).

In Remark~\ref{rmk_realization_product}, we note certain important general cases of product groups where Proposition~\ref{prop_intro_wild_from_tame} can be applied. One important improvement from the earlier known results (see \cite[Theorem~7.5]{Das}) is the following.

\begin{proposition}[{Remark~\ref{rmk_realization_product}~\eqref{item:3}}]
Let $p$ be a prime number. Let $G_1$ and $G_2$ be two perfect quasi $p$-groups such that the PWIC holds for $G_1$ and $G_2$. Then the PWIC is true for $G_1 \times G_2$.
\end{proposition}

When $p=2$, we consider the product of Alternating and Symmetric groups. This is the first non-trivial result towards the PWIC in characteristic $2$. We establish the following evidence towards the GPWIC.

\begin{theorem}[{Corollary~\ref{cor_GPWIC_arbit_product_char_2}}]\label{cor_intro_char_2_all_prod}
When $p = 2$, the GPWIC (Conjecture~\ref{conj_GPWIC}) is true for any product $G = G_1 \times \cdots \times G_n$, $n \geq 1$, where $G_i = A_{d_i}$ for some $d_i \geq 5$ with $4 \nmid d_i$ or $G_i = S_{d_i}$ for an odd integer $d_i \geq 5$.
\end{theorem}

The above result is proved in several steps. After we observe a group theoretic characterization for the potential inertia groups in $A_d$ and $S_d$ covers (Lemma~\ref{lem_wild_candidates}), we prove the PWIC for $A_d$ when $4 \nmid d$, $d \geq 5$ (Theorem~\ref{thm_PWIC_A_d_char_2}) and for $S_d$ when $d \geq 5$ is an odd integer (Theorem~\ref{thm_PWIC_S_d_char_2}). These results are consequence of formal patching results (\cite[Theorem~2.2.3]{Raynaud_AC}), \cite[Theorem~2]{2}); in fact, the covers are constructed out of several simpler Galois covers where the Sylow $2$-subgroups are realized as inertia groups. Further using the patching results to these $A_d$ and $S_d$ covers, one obtains the general Theorem~\ref{cor_intro_char_2_all_prod}.

Throughout this article, we use consequences of formal and rigid patching techniques. Except for the proof of Lemma~\ref{lem_fp_main} (which is essentially the same as \cite[Theorem~2.3.7]{Pries}), the detailed method has not been used in this paper. We omit the technical details and refer to \cite{Ha_St}, \cite[Section 3.4]{Thesis} or \cite{Ha_AC}. 

\subsection*{Acknowledgments}
I would like to thank the referee for various suggestions in improving the presentation of the paper. This work was completed while the author was at Tata Institute of Fundamental Research, Mumbai. The author is supported by NBHM Post-doctoral Fellowship.

\section{Notation, Convention and Definitions}\label{sec_notation}
We will use the following notation throughout this article without repeated mention.

\begin{enumerate}
\item Let $p$ be a prime number. Let $k$ be an algebraically closed field of characteristic $p$.
\item All the $k$-curves considered will be smooth connected curves over $\text{Spec}(k)$, unless otherwise specified.
\item In this article, we work with finite group covers. We will be mostly interested in the Alternating and Symmetric groups of degree $\geq 5$.
\item For a finite group $G$, let $p(G)$ denote the (necessarily normal) subgroup of $G$ generated by all the Sylow $p$-subgroups of $G$. A finite group $G$ is said to be a \textit{quasi {$p$}-group} if $p(G) = G$.
\item For any scheme $X$ and a point $x \in X$, the local ring at $x$ is denoted by $\mathcal{O}_{X,x}$. We denote its completion at the maximal ideal by $\widehat{\mathcal{O}}_{X,x}$. When $\widehat{\mathcal{O}}_{X,x}$ is a domain, $K_{X,x}$ stands for its fraction field.
\end{enumerate}

A \textit{cover} of curves is defined to be a finite, generically separable morphism of $k$-curves. We say that a cover $Y \, \longrightarrow \, X$ is \textit{connected} (respectively, \textit{integral} or \textit{normal}) if $X$ and $Y$ are both \textit{connected} (respectively, integral or normal). For a finite group $G$, a cover $\phi \, \colon \, Y \, \longrightarrow \, X$ of integral $k$-curves is said to be \textit{Galois with group }$G$ or $G$-\textit{Galois} if there is an inclusion of groups $G \hookrightarrow \text{Aut}\left( Y/X \right) \coloneqq \{ \sigma \in \text{Aut}_k(Y) \, | \, f \circ \sigma = f \}$ via which $G$ acts simply transitively on each generic geometric fibre. The Galois theory and the Ramification theory of covers of Noetherian normal schemes is standard in literature, see \cite[Chapter~IV]{Serre_loc}, \cite[Section~3]{Thesis}. Any cover $f \, \colon \, Y \longrightarrow X$ of connected integral $k$-curves is \'{e}tale away from a (possibly empty) finite set of closed points. The closed subset of points in $Y$ where $f$ is not \'{e}tale is said to be the \textit{ramification locus} of $f$, and its image in $X$ is said to be the \textit{branch locus} of $f$.

Our objects of study are the inertia groups over points in a $G$-Galois cover $f \, \colon \, Y \longrightarrow X$ of smooth connected $k$-curves for a finite group $G$. Since $k$ is an algebraically closed field, the inertia group at a point $y \in Y$ is the stabilizer subgroup of $y$ under the natural transitive $G$-action on the finite set $f^{-1}(f(y)) \subset Y$. It turns out that the inertia group at $y$ is also the Galois group of the Galois field extension $K_{Y,y}/K_{X,f(y)}$. Moreover, for $y, \, y' \in Y$ with $f(y) = f(y')$, the inertia groups at $y$ and $y'$ are conjugate in $G$. For a subgroup $I$ of $G$, we say that \textit{{$I$} occurs as an inertia group at a point {$x \in X$}} if there is a point $y \in Y$ with $f(y) = x$ such that $I = \text{Gal}\left( K_{Y,y}/K_{X,x} \right)$. By \cite[Chapter~IV, Corollary~4]{Serre_loc}, every inertia group $I$ is of the form $I = P \rtimes \mathbb{Z}/m$ for some $p$-group $P$ and a co-prime to $p$ integer $m$. We say that the inertia group $I$ is \textit{purely wild} or that \textit{the cover {$f$} is purely wildly ramified over {$x$}} if $I$ is a $p$-group. Similarly, \textit{{$f$} is tamely ramified over {$x$}} or \textit{a tame inertia group {$I$}} refers to the case when $I$ is a cyclic group of order prime-to-$p$. Our main interest in this article is when $X = \mathbb{P}^1_k$.

\begin{definition}\label{def_real_pair}
Let $G$ be a finite group, $I \subset G$. We say that the pair $(G, \, I)$ is \textit{realizable} if there exists a $G$-Galois cover $Y \longrightarrow \mathbb{P}^1$ of smooth projective connected $k$-curves that is \'{e}tale away from $\infty$, and $I$ occurs as an inertia group above $\infty$.
\end{definition}

In the above definition, $G$ is necessarily a quasi $p$-group (i.e. $G = p(G)$), and the group $I$ has the necessary structure of an inertia group.

\section{The GPWIC for certain product groups}
Throughout this section, $k$ is an algebraically closed field of a prime characteristic $p$. Our goal is to establish the Generalized Purely Wild Inertia Conjecture (the GPWIC, Conjecture \ref{conj_GPWIC}) for any finite product of Alternating groups for an arbitrary prime $p$ and also of Symmetric groups when $p=2$. Recall that the GPWIC (Conjecture~\ref{conj_GPWIC}) asserts that given a quasi $p$-group $G$ and a nonempty finite set $B \, \subset \, \mathbb{P}^1$ of closed points, there is connected $G$-Galois cover of $\mathbb{P}^1_k$, \'{e}tale away from $B$, with prescribed $p$-groups (satisfying a necessary condition: their conjugates generate $G$) as the inertia groups above the points in $B$.

When $p$ is an odd prime number, $A_d$ is a quasi $p$-group for all $d \geq \text{\rm max}\{5, p\}$. For $p=2$, \, $A_d$ is a quasi $2$-group for $d \geq 5$, and $S_d$ is a quasi $2$-group for $d \geq 3$. Throughout this section, we will implicitly assume these bounds on $d$. Our strategy is to first prove the PWIC (Conjecture \ref{conj_PWIC}) for $A_d$ (with some restrictions on $d$ for $p=2$), namely, for any $p$-subgroup $P \subset A_d$, there is a connected $A_d$-Galois cover of $\mathbb{P}^1$, \'{e}tale away from $\infty$, such that $P$ occurs as an inertia group above $\infty$ (i.e. the pair $(A_d, \, P)$ is realizable; Definition~\ref{def_real_pair}). Then the argument of \cite[Theorem~7.4]{DK} shows that the GPWIC holds for arbitrary product of the above Alternating groups. We will prove the GPWIC for product of Symmetric groups in Theorem~\ref{thm_product_S_d_char_2}.

\subsection{Odd Characteristic Case}\label{sec_odd}
We suppose that $p \geq 3$. By \cite[Theorem 1.7(3)]{Das}, the GPWIC is true for any product $A_{d_1} \times \cdots \times A_{d_n}$, $n \geq 1$, where each $d_i = p$ or $d_i \geq p+1$ is co-prime-to $p$. This was proved using formal patching techniques and by understanding the covers of the affine $k$-line given by explicit affine equations, \`{a} la Abhyankar. We will remove restrictions on the $d_i$'s:
$$p \nmid d_i \text{ if } d_i > p$$
which arose from the fact that the status of the PWIC remained unknown for the groups $A_{rp}$, $r \geq 2$ (see \cite[Corollary~5.5, Corollary~5.6]{DK}). This is achieved as an application of the general result Proposition~\ref{prop_purely_wild_from_tame} which in turn is obtained as a consequence of the following result on the construction of covers using formal patching due to Pries (\cite[Theorem~2.3.7, Remark~2.3.8]{Pries}; although the original result assumes that $p$ strictly divides the order of the group $G$ in its statement, we observe that the same argument works in the following general set up as the inertia groups have Sylow $p$-subgroups of order $p$). Although the proof is a bit technical, involving certain numerical criteria followed by a standard application of formal patching and a Lefchetz type argument, we include it for the sake of completeness.

\begin{lemma}[{\cite[Theorem~2.3.7, Remark~2.3.8]{Pries}}]\label{lem_fp_main}
Let $X$ be a smooth projective connected $k$-curve, $x \in X$ be a closed point. Let $G$ be a finite group and $I \subset G$ be an extension of a $p$-cyclic group $\langle \, \tau \, \rangle \cong \mathbb{Z}/p$ by a cyclic group $\langle c \rangle$ of order $m$, $(m,p) \, = \, 1$. Let $G_1$ and $G_2$ be subgroups of $G$ such that $\tau \in G_1, \, I \subset G_2$, and assume that the following hold.
\begin{enumerate}
\item There is a connected $G_1$-Galois cover $\psi_1$ of $X$ \'{e}tale away from a set $B \subset X$ of closed points with $x \in B$. Suppose that $\langle \, \tau \, \rangle$ occurs as an inertia group above $x$. For each point $y \neq x$ in $B$, let $I_y$ denote an inertia group above $y$.
\item $\psi_2 \, \colon \, Y_2 \longrightarrow \mathbb{P}^1$ is a connected $G_2$-Galois cover, \'{e}tale away from a set $\{0 = \eta_0, \eta_1, \ldots, \eta_r \}$ such that $I$ occurs as an inertia group above $0$, and for $1 \leq i \leq r$, let $J_i$ denote an inertia group above $\eta_i$.
\end{enumerate}

If $G = \langle \, G_1,  G_2 \, \rangle$, there is a set $B' =  \{ x_1, \ldots, x_r \}$ of closed points in $X$, disjoint from $B$, and a connected $G$-Galois cover $Y \longrightarrow X$, \'{e}tale away from $B \sqcup B'$, such that $I$ occurs as an inertia group above $x$, for each point $y \neq x$ in $B$, $I_y$ occurs as an inertia group above $y$, and $J_i$ occurs as an inertia group above $x_i$ for $1 \leq i \leq r$.
\end{lemma}

\begin{proof}
We have $I = \langle \, \tau \, \rangle \rtimes \langle \, c \, \rangle$. Let $m'$ be the order of the prime-to-$p$ part of the center of $I$. Let $u$ be a local parameter in $\mathbb{P}^1$ at $0$. Let $y_2 \in Y_2$ be a point above $0$ such that $I = \text{Gal}\left( K_{Y_2,y_2}/k((u)) \right)$. Let $h_2$ be the conductor of this extension. Then $(h_2,p)=1$, $(h_2,m)=m'$, and $m \, | \, h_2(p-1)$. Also let $v$ be a local parameter in $X$ at $x$. Thus we identify $K_{X,x}$ with $k((v))$. Suppose that $h_1'$ be the conductor for the local $\langle \, \tau \, \rangle \cong \mathbb{Z}/p$-Galois field extension in $\psi_1$ near $x$. So, $(h_1',p)=1$.

Choose $\gamma$ co-prime to $p$ such that $(\gamma,m)=1$, $p \nmid (\gamma m +1)$, and $\gamma h_2 \geq h_1'$. Set $h_1 \coloneqq \gamma h_2$. By \cite[Theorem 2.2.2]{Pries}, there is a connected $G_1$-Galois cover $\phi_1 \, \colon Y_1 \longrightarrow X$ with the same inertia groups and ramification behavior above the points $y \neq x$ in $X$, and there is a point $y_1 \in Y_1$ above $x$ such that the $\langle \, \tau \, \rangle$-Galois field extension $K_{Y_1,y_1}/K_{X,x}$ has conductor $h_1$.

Consider the integral $k[[t]]$-scheme $S \, \coloneqq \text{Spec}\left( k[[u,v,t]]/(uv-t^{\gamma m +1}) \right)$, and denote its closed fibre (the subscheme given by the locus of $t=0$) by $S'$. By our choice of the local parameters, $S'$ is the union of $\text{Spec}\left( \widehat{\mathcal{O}}_{X,x} \right)$ with $\text{Spec} \left( \widehat{\mathcal{O}}_{\mathbb{P}^1,0} \right)$, meeting at the point $\sigma = (x,0)$ which we identify with $(v=0, u=0)$ and has equation $u v = t^{\gamma m +1}$ near the point $\sigma$. 

Set $e \coloneqq h_1 m + h_2 = (\gamma m +1) h_2$. Then $(e,m) = (h_2,m) =  m'$, \, $p \nmid e$, and $(m,\frac{e}{(h_1,h_2)}) = (m , \frac{e}{h_2}) = (m, \gamma m +1) = 1$, and \cite[Notation 2.3.2 and Numerical Hypothesis]{Pries} hold. By \cite[Theorem 2.3.4]{Pries}, there is an irreducible $I$-Galois cover $\widehat{\phi} \, \colon \, \widehat{Z} \longrightarrow S$ whose generic fibre is irreducible, and the special fibre $\widehat{\phi}_k \, \colon \, Z_k \longrightarrow S'$ has the following properties.
\begin{eqnarray*}
Z_k \times_{S'} \text{Spec}\left( K_{X,x} \right) & \cong & \text{Ind}_{\langle \, \tau \, \rangle}^I \, \text{Spec}\left( K_{Y_1,y_1} \right),\\
Z_k \times_{S'} \text{Spec}\left( k((u)) \right) & \cong & \text{Spec}\left( K_{Y_2,y_2} \right)
\end{eqnarray*}
as $I$-Galois \'{e}tale covers of $\text{Spec}\left( K_{X,x} \right)$ and $\text{Spec}\left( k((u)) \right)$, respectively.

Let $T$ be a regular irreducible projective $k[[t]]$-curve whose geometric fibre is $X_{k((t))}$, and the closed fibre is $X \cup \mathbb{P}^1$ meeting at the point $\sigma = (x,0)$ (which we have identified with $(v=0, u=0)$ having equation $u v = t^{\gamma m +1}$ near $\sigma$; so $\text{Spec}\left( \widehat{\mathcal{O}}_{T, \sigma} \right) = S$). Let $X - x = \text{Spec}\left( A \right)$, \, $Y_1 - \psi_1^{-1}(x) = \text{Spec}\left( B_1 \right)$, \, $Y_2 - \psi_2^{-1}(0) = \text{Spec}\left( B_2 \right)$. By \cite[Proposition 2.3]{Ha_AC}, there is an irreducible normal $\langle \, G_1, G_2 \, \rangle = G$-Galois cover $h \, \colon \, V \longrightarrow T$ such that
\begin{eqnarray*}
V \, \times_T \, \text{Spec}\left( A[[t]] \right) & \cong & \text{Ind}_{G_1}^G \, \text{Spec}\left( B_1[[t]] \right),\\
V \, \times_T \, \text{Spec}\left( k[u^{-1}][[t]] \right) & \cong & \text{Ind}_{G_2}^G \, \text{Spec}\left( B_2[[t]] \right), \, \text{and}\\
V \, \times_T \, \text{Spec}\left( \widehat{\mathcal{O}_{T,\sigma}} \right) & \cong & \text{Ind}_{I}^G \, \text{Spec}\left( \widehat{Z} \right)
\end{eqnarray*}
as Galois covers of $\text{Spec}\left( A[[t]] \right)$, \, $\text{Spec}\left( k[u^{-1}][[t]] \right)$ and $S = \text{Spec}\left( \widehat{\mathcal{O}}_{T,\sigma} \right)$, respectively. Let $h^0 \, \colon \, V^0 \longrightarrow X_{k((t))}$ be the generic fibre of $h$. Then there exists a set $B' = \{x_1, \ldots, x_r\} \subset X$ disjoint from $B$ such that $h^0$ is \'{e}tale away from $\{b_{k((t))} \, | \, b \in B \sqcup B'\}$,\, $I$ occurs as an inertia group above $x_{k((t))}$, for each point $y \neq x$ in $B$, $I_y$ occurs as an inertia group above $y_{k((t))}$, for each $1 \leq i \leq r$, $J_i$ occurs as an inertia group above $x_{i,k((t))}$. As $h$ is a cover of smooth irreducible $k((t))$-curves and the closed fibre of $h$ is generically smooth, the result follows from \cite[Corollary 2.7]{Ha_AC}.
\end{proof}

As a consequence of the above, we have the following characteristic-free general result on the realization of certain pairs (Definition~\ref{def_real_pair}). This can be applied to potentially increase the Galois group without changing the $\mathbb{Z}/p$ inertia.

\begin{proposition}\label{prop_purely_wild_from_tame}
Let $p$ be a prime. Let $G$ be a finite group, $G_1 \times G_2 \subset G$, where $G_1$ and $G_2$ are quasi $p$-groups. Suppose that for $i = 1, \, 2$, $\lambda_i$ is an element of order $p$ in $G_i$ such that the pair $(G_1 \times G_2, \, \langle \, (\lambda_1, \lambda_2) \, \rangle)$ is realizable. Let $c \in G$ be an element of order prime-to-$p$ that interchanges $\lambda_1$ and $\lambda_2$ via conjugation, i.e. $c^{-1} \, (\lambda_1, \lambda_2) \, c = (\lambda_2, \lambda_1)$. Consider the subgroup $H \coloneqq \langle \, G_1 \times G_2 , \, \langle \, c \, \rangle \, \rangle$ of $G$. Let $T$ be the maximal common quotient of $H$ and $\langle \, c \, \rangle$. Then the pair $(H \times_T \langle \, c \, \rangle, \, \langle \, (\lambda_1, \lambda_2) \, \rangle)$ is realizable.
\end{proposition}

\begin{proof} 
Let $Y_1 \longrightarrow \mathbb{P}^1$ be a connected $G_1 \times G_2$-Galois cover, \'{e}tale away from $\infty$, such that $\langle \, (\lambda_1, \lambda_2) \, \rangle$ occurs as an inertia group above $\infty$. Let $Y_2 \longrightarrow \mathbb{P}^1$ be a $\langle \, (\lambda_1, \lambda_2) \, \rangle \rtimes \langle \, c \, \rangle$-Galois Harbater-Katz-Gabber cover that is totally ramified over $\infty$, \, $\langle \, c \, \rangle$ occurs as an inertia group above $0$, and is \'{e}tale everywhere else. By Lemma~\ref{lem_fp_main}, we obtain a connected $H$-Galois cover $Y \longrightarrow \mathbb{P}^1$, \'{e}tale away from $\{0, \infty\}$, such that $\langle \, (\lambda_1, \lambda_2) \, \rangle \rtimes \langle \, c \, \rangle$ occurs as an inertia group above $\infty$, and $\langle \, c \, \rangle$ occurs as an inertia group above $0$. Consider the connected $\langle \, c \, \rangle$-Galois Kummer cover $\psi \, \colon \, Z \cong \mathbb{P}^1 \longrightarrow \mathbb{P}^1$ that is \'{e}tale away from $\{0, \infty\}$, over which the cover is totally ramified. Let $W$ be a dominant component in the normalization of $Y \times_{\mathbb{P}^1} Z$, and consider the cover $f \colon W \longrightarrow Z \cong \mathbb{P}^1$. Since the function field $k(W)$ is the compositum $k(Y) \cdot k(Z)$, we see that $f$ is a Galois cover with group $H \times_T \langle \, c \, \rangle$. By \cite[Proposition~3.5]{Manish_Killing}, the Galois cover $f$ is \'{e}tale away from $\infty$, and $\langle \, (\lambda_1, \lambda_2) \, \rangle$ occurs as an inertia group above $\infty$.
\end{proof}

Now we are ready to prove the PWIC (Conjecture~\ref{conj_PWIC}) for $A_{rp}$.

\begin{theorem}\label{thm_PWIC_Alternating}
Let $p$ be an odd prime, $r \geq 1$. Then the PWIC holds for $A_{rp}$.
\end{theorem}

\begin{proof}
By \cite[Theorem~1.2]{BP}, the IC -- Conjecture~\ref{conj_IC} (and hence the PWIC) holds for $A_p$ with $p \geq 5$. The PWIC holds for $A_3$ since $A_3$ is the $3$-cyclic group. We suppose that $r \geq 2$.

For $1 \leq i \leq r$, consider the $p$-cycle
$$\tau_i \coloneqq ((i-1)p+1, \ldots, ip) \in A_{rp}.$$
Since $A_{rp}$ is a simple quasi $p$-group, it is generated by the conjugates of any cyclic subgroup of order $p$. As any $p$-subgroup of $A_{rp}$ contains an element of order $p$, and the inertia groups above a point in any connected $A_{rp}$-Galois cover are conjugates in $A_{rp}$, in view of \cite[Theorem~2]{2}, it is enough to prove the following.

\emph{For each $1 \leq u \leq r$, the pair $(A_{rp}, \, \langle \, \tau_1 \cdots \tau_u \, \rangle)$ is realizable.}

When $u < r$, this is \cite[Corollary~5.6]{DK}. Let $u=r$, and set $\tau = \tau_1 \cdots \tau_r$. We proceed via induction on $r$ and apply Proposition~\ref{prop_purely_wild_from_tame} to show that the pair $(A_{rp}, \, \langle \, \tau \, \rangle)$ is realizable. By the induction hypothesis, the pair $(A_{(r-1)p}, \, \langle \, \tau_1 \cdots \tau_{r-1} \, \rangle)$ is realizable. By \cite[Theorem~5.2]{DK}, the pair $(A_{(r-1)p} \times A_p, \, \langle \, (\tau_1 \cdots \tau_{r-1}, \tau_r) \, \rangle)$ is realizable, as well. Now take $G_1 = A_{(r-1)p}, \, G_2 = A_p = \text{Alt}\{(r-1)p+1, \ldots, rp\}, \, G = S_{rp}, \, \lambda_1 = \tau_1 \cdots \tau_{r-1}$ and $\lambda_2 = \tau_r$ in Proposition~\ref{prop_purely_wild_from_tame}. Consider the odd permutation
$$c \coloneqq (1, (r-1)p+1)\,(2, (r-1)p+2) \, \cdots \, (p, rp) \in S_{rp}$$
of order $2$. Under the natural embedding $A_{(r-1)p} \times A_p \hookrightarrow S_{rp}$, we identify the element $(\tau_1 \cdots \tau_{r-1}, \tau_r)$ with $\tau = \tau_1 \cdots \tau_r$. The element $c$ acts on $\tau$ via conjugation. Then we have $H = \langle \, G_1 \times G_2 , \, \langle \, c \, \rangle \, \rangle = S_{rp}$. As the maximal common quotient of $H$ and $\langle \, c \, \rangle \cong \mathbb{Z}/2$ is $\langle \, c \, \rangle$, by Proposition~\ref{prop_purely_wild_from_tame}, we obtain a connected $A_{rp}$-Galois cover of $\mathbb{P}^1$, branched only at $\infty$, and $\langle \, \tau \, \rangle$ occurs as an inertia group above $\infty$.
\end{proof}

A similar argument as in \cite[Theorem~7.4]{Das} establishes the GWPIC for product of Alternating groups.

\begin{corollary}\label{cor_GPWIC_Alt_products}
Let $p$ be an odd prime, $n \geq 1$. For $1 \leq i \leq n$, let $d_i \geq \text{max}\{p, 5\}$. The GPWIC holds for the product of Alternating groups $A_{d_1} \times \cdots \times A_{d_n}$.
\end{corollary}

We note the following cases where Proposition~\ref{prop_purely_wild_from_tame} can be applied.

\begin{remark}\label{rmk_realization_product}
The pair $(G_1 \times G_2, \, \langle \, (\tau_1, \tau_2) \, \rangle)$ is realizable in the following cases.
\begin{enumerate}
\item (\cite[Theorem~5.2]{DK}) When the PWIC holds for two perfect (groups whose derived subgroup is the whole group) quasi $p$-group $G_1$ and $G_2$, and for any $1 \leq a \leq p-1$, there is an automorphism of $G_1 \times G_2$ taking $(\tau_1^a, \tau_2)$ to $(\tau_1, \tau_2)$, the pair $(G_1 \times G_2, \, \langle \, (\tau_1, \tau_2) \, \rangle)$ is realizable.
\item (\cite[Theorem~7.5]{Das}) Suppose that for $i = 1, \, 2$, the pair $(G_i, \, \langle \, \tau_i \, \rangle)$ is realizable. Let $\pi_i \, \colon \, G_1 \times G_2 \longrightarrow G_i$ be the projections. If $G_1$ and $G_2$ do not have a common quotient, and $Q \subset G_1 \times G_2$ is a $p$-subgroup such that $\pi_i(Q) \cong \langle \, \tau_i \, \rangle$, then the pair $(G_1 \times G_2, \, Q)$ is realizable. In particular, this holds for $Q = \langle \, (\tau_1, \tau_2) \, \rangle$.\label{item:2}
\item We note that the above result \eqref{item:2} can be generalized to:

\emph{for perfect quasi $p$-groups $G_1$ and $G_2$ for which the PWIC hold, the PWIC also holds for $G_1 \times G_2$.}

To see this, consider a $p$-subgroup $P \subset G \coloneqq G_1 \times G_2$ with $\langle \, P^G \, \rangle = G$. Let $\pi_i$ be the obvious projections as before. By Goursat's Lemma, $P = \pi_1(P) \times_Q \pi_2(P)$ for a common quotient $Q$ of $\pi_1(P)$ and $\pi_2(P)$. By \cite[Corollary~4.6]{Manish_Compositum}, the pairs $(G_1 \times \pi_2(P), P)$ and $(\pi_1 \times G_2, P)$ are realizable. Applying \cite[Lemma~4.6]{Das} to the Galois covers realizing these pairs, we obtain respective covers with the following additional property: the local $P$-Galois extensions over $K_{\mathbb{P}^1,\infty}$ are isomorphic. This allows us to use a standard formal patching technique to obtain a connected $G_1 \times G_2$-Galois cover of $\mathbb{P}^1$, \'{e}tale away from $\infty$, such that $P$ occurs as an inertia group above $\infty$.\label{item:3}
\end{enumerate}
\end{remark}

Before moving on to the even characteristic case, we see another consequence of the patching result Lemma~\ref{lem_fp_main} and the above corollary. This allows one to check the full Inertia Conjecture for a fewer candidates of potential inertia groups.

\begin{proposition}\label{prop_reduction_result}
Let $p$ be a prime number, $G$ be a finite group, and $\mathbb{Z}/p \cong P$ be a subgroup of $G$ such that the pair $(G, \, P)$ is realizable. Let $\beta \in N_G(P)$ be an element of order prime-to-$p$. Suppose that $H$ is a subgroup of $G$ containing $I \coloneqq P \rtimes \langle \, \beta \, \rangle$, and the pair $(H, \, I)$ is realizable. Then for any $p$-subgroup $P'$ of $G$ containing $P$ that is normalized by $\beta$, the pair $(G, \, P' \rtimes \langle \, \beta \, \rangle)$ is realizable.

In particular, let $p$ be an odd prime, $d' \geq d \geq p$. Suppose that for some element $\beta \in N_{A_d}(\langle \, (1, \ldots, p) \, \rangle)$ of order prime-to-$p$, the pair $(A_d, \, \langle \, (1, \ldots, p) \, \rangle \rtimes \langle \, \beta \, \rangle)$ is realizable. If $P \subset A_{d'}$ is a $p$-subgroup containing the $p$-cycle $(1, \ldots, p)$ that is normalized by $\beta$, the pair $(A_{d'}, \, P  \rtimes \langle \, \beta \, \rangle)$ is realizable.
\end{proposition}

\begin{proof}
By our assumption, there are connected Galois covers $\psi_1$ and $\psi_2$ of $\mathbb{P}^1$, \'{e}tale away from $\infty$, with Galois groups $G$ and $H$, respectively, and such that the inertia groups above $\infty$ are the conjugates of $P$ and $I$, respectively, in the corresponding Galois groups. By Lemma~\ref{lem_fp_main}, the pair $(G, \, I)$ is realizable. By \cite[Theorem 2]{2}, the pair $(G, \, P' \rtimes \langle \, \beta \, \rangle)$ is realizable as well.

By Corollary~\ref{cor_GPWIC_Alt_products}, for any $d' \geq p$, the pair $(A_{d'}, \, \langle \, (1, \ldots, p) \, \rangle )$ is realizable. So, the first statement applies with $G = A_{d'}$ and $H = A_d \subset A_{d'}$.
\end{proof}

Using the above, we prove the full Inertia Conjecture (Conjecture~\ref{conj_IC}) for $A_{p+1}$ for $p \geq 5$. This was proved in \cite[Theorem~5.3]{Das} under the assumption that $p \equiv 2 \pmod{3}$.

\begin{corollary}\label{cor_A_p+1}
For any prime number $p \geq 5$, the Inertia Conjecture is true for $A_{p+1}$.
\end{corollary}

\begin{proof}
This follows from Proposition~\ref{prop_reduction_result} since
$$N_{A_p}( \langle \, (1, \cdots, p) \, \rangle ) = N_{A_{p+1}}( \langle \, (1, \cdots, p) \, \rangle ),$$
and the Inertia Conjecture (Conjecture~\ref{conj_IC}) holds for $A_p$, $p \geq 5$ by \cite[Theorem~1.2]{BP}.
\end{proof}

\subsection{Even characteristic case}\label{sec_two}
Now we show that the Generalized Purely Wild Inertia Conjecture (GPWIC, Conjecture~\ref{conj_GPWIC}) is true for any product of certain Symmetric and Alternating groups in characteristic $2$. This is the first ever non-trivial result towards the PWIC and the GPWIC in characteristic $2$. One of the main ingredients of the proofs will be Raynaud's patching result \cite[Theorem~2.2.3]{Raynaud_AC}. We start with the study of the PWIC.

Let $P$ be a $2$-subgroup of a finite quasi $2$-group $G$ such that $\langle \, P^G \, \rangle \, = \, G$, i.e. the conjugates of $P$ in $G$ generate $G$. The PWIC asserts that the pair $(G, \, P)$ is realizable (Definition~\ref{def_real_pair}). By \cite[Theorem 2]{2}, if $P$ is a $2$-subgroup as above for which the pair $(G, \, P)$ is realizable, and $P' \subset G$ is a $2$-subgroup containing $P$, then the pair $(G, \, P')$ is also realizable. Note that a Symmetric group $S_d$ is a quasi $2$-group for all $d \geq 2$ ($S_3$ has order strictly divisible by $2$, and the PWIC holds for it by \cite[Corollary~2.2.2]{Raynaud_AC}). In fact, for $d \geq 2$, \, $S_d$ is a quasi $p$-group only for the prime $p \, = \, 2$. On the other hand, $A_d$ is a quasi $2$-group for all $d \geq 5$ ($A_3$ and $A_4$ are quasi $p$-groups only for the prime $p \, = \, 3$). In these cases, we characterize the potential purely wild inertia groups ($2$-subgroups) $P$.

\begin{lemma}\label{lem_wild_candidates}
Let $d \geq 5$. We have the following.
\begin{enumerate}
\item Let $P \subset S_d$ be a $2$-subgroup such that $\big\langle \, P^{\, S_d} \, \big\rangle \, = \, S_d$. Then $P$ contains an odd permutation $\tau$ such that $\big\langle \, \langle \, \tau \, \rangle^{\, S_d} \, \big\rangle \, = \, S_d$.\label{i:1}
\item Let $P \subset A_d$ be a $2$-subgroup such that $\big\langle \, P^{\, A_d} \, \big\rangle \, = \, A_d$. Then $P$ contains an even permutation $\tau$ of order $2$ such that $\big\langle \, \langle \, \tau \, \rangle^{\, A_d} \, \big\rangle \, = \, A_d$.\label{i:2}
\end{enumerate}
If $P \subset S_4$ is a $2$-subgroup such that $\big\langle \, P^{ \, S_4} \, \big\rangle \, = \, S_4$, then $P$ contains either a transposition or a $4$-cycle $\tau$ such that $\big\langle \, \langle \, \tau \, \rangle^{\, S_4} \, \big\rangle \, = \, S_4$.
\end{lemma}

\begin{proof}
\begin{enumerate}
\item If $P \subset A_d$, every conjugate of $P$ is also contained in $A_d$. Since $\big\langle \, P^{\, S_d} \, \big\rangle \, = \, S_d$, we may choose an odd permutation $\tau \in P$. Since $\big\langle \, \langle \, \tau \, \rangle^{\, S_d} \, \big\rangle$ is a non-trivial normal subgroup of $S_d$ containing an odd permutation $\tau$, $\big\langle \, \langle \, \tau \, \rangle^{\, S_d} \, \big\rangle \, = \, S_d$.
\item Since $A_d$ is a simple group, for every non-identity element $\tau \in A_d$, we have $\big\langle \, \langle \, \tau \, \rangle^{\, A_d} \, \big\rangle \, = \, A_d$. As $P$ always contains an element $\tau$ of order $2$, the result follows.
\end{enumerate}
As in \eqref{i:1}, $P$ contains an odd permutation $\tau \in S_4$ such that $\big\langle \, \langle \, \tau \, \rangle^{\, S_4} \, \big\rangle \, = \, S_4$. In $S_4$, such an element must be either a transposition or a $4$-cycle.
\end{proof}

We first consider connected $A_d$-Galois and $S_d$-Galois covers ($d \geq 5$) of $\mathbb{P}^1$ over an algebraically closed field of characteristic $2$, \'{e}tale away from $\{\infty\}$, such that the inertia groups above $\infty$ are cyclic of order $2$; cf. Lemma~\ref{lem_wild_candidates}.

\begin{proposition}\label{prop_cyclic_order_two_inertia}
Let $d \geq 5$, and $1 \leq r < \floor{d/2}$ be an integer. For $1 \leq u \leq r$, let $\tau_u$ be the transposition $(2u-1, 2u)$, and consider the element $\tau \coloneqq \tau_1 \cdots \tau_r \in S_d$ of order $2$. Over any algebraically field of characteristic $2$, we have the following.
\begin{enumerate}
\item If $r$ is an even integer, the pair $(A_d, \, \langle \, \tau \, \rangle)$ is realizable.\label{Alt}
\item If $r$ is an odd integer, the pair $(S_d, \, \langle \, \tau \, \rangle)$ is realizable.\label{Sym}
\end{enumerate}
\end{proposition}

\begin{proof}
For each $1 \leq u \leq r$ and each $b \in \{2r+1, \ldots , d\}$ (as $2r < d$ by out assumption, this set is non-empty), consider the $3$-cycle
$$\sigma_{u,b} \, \coloneqq \, (2u-1,2u, b).$$
Then for any $u$ and $b$, we have $\tau_u^{-1} \sigma_{u,b} \tau_u = \sigma_{u,b}^2$. As $\tau_{u'}$ has support disjoint from the set $\{2u-1, \, 2u, \, b\}$ for $u' \neq u$, $\tau^{-1} \sigma_{u,b} \tau = \sigma_{u,b}^2$. So for each $u$ and $b$, $\tau \in N_{S_d}(\langle \, \sigma_{u,b} \, \rangle)$, and the subgroup $H_{u,b} \coloneqq \langle \, \sigma_{i,b} \, \rangle \rtimes \langle \, \tau \, \rangle$ is a quasi $2$-group of order $6$. As $\langle \, \tau \, \rangle$ is a Sylow $2$-subgroup of $H_{u,b}$, by Raynaud's proof of the Abhyankar's Conjecture on the affine line (\cite[Corollary~2.2.2]{Raynaud_AC}), the pair $(H_{u,b}, \, \langle \, \tau \, \rangle)$ is realizable when the base field has characteristic $2$. Consider the quasi $2$-subgroup
$$H \, \coloneqq \, \langle \, H_{u,b} \, | \, 1 \leq u \leq r, \, 2r+1 \leq b \leq d \, \rangle$$
of $S_d$. From the construction, $H$ is a primitive subgroup of $S_d$. By the patching result \cite[Theorem~2.2.3]{Raynaud_AC}, the pair $(H, \, \langle \, \tau \, \rangle)$ is realizable.

When $d = 5$, by our assumption, $r = 1 \text{ or } 2$. If $r = 1$, $H$ contains the transposition $(1,  2)$. By Jordan's Theorem \cite[Theorem~1.1]{Jones}, $A_5 \subset H$. If $r = 2$, $H$ contains the $3$-cycle $(1,2,5)$ that fixes $2$ points in $\{1, \ldots, 5\}$. By \cite[Theorem~1.2(3)]{Jones}, $A_5 \subset H$ (note that $PGL_2(4) \cong A_5$). For $d \geq 6$, $H$ is a primitive permutation group containing a $3$-cycle $(1,2,d)$; so again by Jordan's Theorem, $A_d \subset H$. Since for all $d \geq 5$, $H$ is a quasi $2$-subgroup of $S_d$ containing $A_d$, we have $H \in \{A_d, S_d\}$.

If $r$ is an even integer, $\tau$ is an even permutation. So $\left\langle \, \langle \, \tau \, \rangle^{\, S_d} \, \right\rangle = A_d$. Since $\left\langle \, \langle \, \tau \, \rangle^{\, H} \, \right\rangle = H$, we have $H = A_d$. If $r$ is an odd integer, $\tau$ is an odd permutation and we have $H = S_d$.
\end{proof}

\begin{remark}\label{rmk_S_4_transposition}
A similar argument as above shows that for any transposition $\tau \in S_4$, the pair $(S_4, \, \langle \, \tau \, \rangle)$ is realizable. The subgroup $H$ as in the proof becomes a primitive quasi $2$-subgroup of $S_4$ containing the transposition $\tau$ and a $3$-cycle. As $PGL_2(3) \cong S_4$, by \cite[Theorem~1.2(3)]{Jones}, $H$ contains $A_4$. Since $H$ is a quasi $2$-subgroup of $S_4$, we conclude that $H = S_4$.
\end{remark}

\begin{theorem}\label{thm_PWIC_A_d_char_2}
Let $d \geq 5$ be an integer that is not a multiple of $4$. The PWIC is true for $A_d$ in characteristic $2$.
\end{theorem}

\begin{proof}
Let $P \subset A_d$ be a $2$-subgroup such that $\left\langle \, P^{\, A_d} \, \right\rangle = A_d$. By Lemma~\ref{lem_wild_candidates}~\eqref{i:2}, there is an even permutation $\tau \in P$ of order $2$ such that $\left\langle \, \langle \, \tau \, \rangle^{\, A_d} \, \right\rangle = A_d$. Consider a disjoint cycle decomposition $\tau = \tau_1 \cdots \tau_r$ of $\tau$ in $A_d$ where each $\tau_u$ is a transposition, and $r$ is necessarily an even integer, $2r \leq d$. Since $4 \nmid d$ by our assumption, $r < \floor{d/2}$. As $r \geq 2$, $\tau$ is conjugate to the permutation $(1, 2) \cdots (2r-1, 2r)$. By Proposition~\ref{prop_cyclic_order_two_inertia}~\eqref{Alt}, the pair $(A_d, \, \langle \, \tau \, \rangle)$ is realizable. Now applying \cite[Theorem~2]{2}, the pair $(A_d, \, P)$ is realizable.
\end{proof}

Remark~\ref{rmk_realization_product}\eqref{item:3} and the above result have the following immediate consequence.

\begin{corollary}\label{cor_GPWIC_product_A_d_char_2}
In characteristic $2$, the GPWIC (Conjecture~\ref{conj_GPWIC}) is true for any product $A_{d_1} \times \cdots \times A_{d_u}$, $u \geq 1$, where each $d_i \geq 5$ is an integer not divisible by $4$. 
\end{corollary}

\begin{remark}
The smallest case the above theorem does not include is $A_8$. When $\tau = (1, 2)(3, 4)(5, 6)(7, 8) \in A_8$, the realization of the potential purely wild inertia group $\langle \, \tau \, \rangle$ remains unknown. Note that when $4 | d$, $d \geq 5$, $r = d/2$, and $\tau = (1,2) \cdots (2r-1,r)$, proving that the pair $(A_d, \, \langle \, \tau \, \rangle)$ is realizable would imply that the PWIC holds for $A_d$ (with any $d \geq 5$) in characteristic $2$.
\end{remark}

\begin{theorem}\label{thm_PWIC_S_d_char_2}
In characteristic $2$, the PWIC is true for $S_2$, $S_3$, and $S_d$ for all odd integer $d \geq 5$.
\end{theorem}

\begin{proof}
By \cite[Corollary~2.2.2]{Raynaud_AC}, the PWIC in characteristic $2$ is true for the groups $S_2$ and $S_3$. Let $d \geq 2$ be an odd integer. Let $P \subset S_d$ be a $2$-subgroup such that $\left\langle \, P^{\, S_d} \, \right\rangle = S_d$. Using Lemma~\ref{lem_wild_candidates}~\eqref{i:1}, choose an odd permutation $\tau \in P$ such that $\left\langle \, \langle \, \tau \, \rangle^{\, S_d} \, \right\rangle = S_d$. By \cite[Theorem~2]{2}, it is enough to show that the pair $(S_d, \, \langle \, \tau \, \rangle)$ is realizable. Suppose that $\text{ord}(\tau) = 2^f$, $f \geq 1$.

If $\tau$ has order $2$, then $\tau = \tau_1 \cdots \tau_r$ for disjoint transpositions $\tau_u$'s, for some odd integer $r \leq \floor{d/2}$. As $d$ is an odd integer, $2r < d$. Moreover, $\tau$ is conjugate to the element $(1,2) \cdots (2r-1,2r)$. By Proposition~\ref{prop_cyclic_order_two_inertia}~\eqref{i:1}, the pair $(S_d, \, \langle \, \tau \, \rangle)$ is realizable.

Let $\text{ord}(\tau) \geq 4$. Then $\tau^2$ is a non-trivial even permutation. By Theorem~\ref{thm_PWIC_A_d_char_2}, the pair $(A_d, \, \langle \, \tau^2 \, \rangle)$ is realizable. Also the pair $(\langle \, \tau \, \rangle, \, \langle \, \tau \, \rangle)$ is realizable as the PWIC holds for any $2$-group. By \cite[Theorem~2.2.3]{Raynaud_AC}, the pair $(S_d = \langle \, A_d, \langle \, \tau \, \rangle \, \rangle, \, \langle \, \tau \, \rangle)$ is realizable.
\end{proof}

\begin{example}[{Failure of Abhyankar-style equation to the PWIC for $S_4$ in characteristic $2$}]
The smallest case not covered in the above result is $d = 4$. One can construct the connected $A_d$-Galois and $S_d$-Galois covers using covers $\mathbb{P}^1 \longrightarrow \mathbb{P}^1$ given by explicit affine equations, following Abhyankar's work or \cite[Section~3]{Das}. It is not clear whether we can use this method to obtain a connected $S_4$-Galois cover \'{e}tale cover of $\mathbb{A}^1$ such that $\langle \, (1,2,3,4) \, \rangle$ occurs as an inertia group above $\infty$. We illustrate one such example in the following.

Let $\psi \, \colon \, Y = \mathbb{P}^1_y \longrightarrow \mathbb{P}^1_x$ be a degree-$4$ cover given by the affine equation $f(x,y) = 0$, where
$$f(x,y) = y^4 - y^3 +x.$$
Since $(0,0)$ is the only common zero of $f$ and its $y$-derivative $\frac{\partial f}{\partial y}$, the cover $\psi$ is \'{e}tale away from $\{0, \infty\}$. From the equation, it follows that there are two points $(y=0)$ and $(y=1)$ above $x=0$ in $\psi$ having ramification indices $3$ and $1$, respectively, and $\psi^{-1}(\infty)$ consists of the unique point $(y= \infty)$ having ramification index $4$. Let $\phi \, \colon \, Z \longrightarrow \mathbb{P}^1_x$ be the Galois closure of $\psi$, and let $G$ be the Galois group. Since $f(x,y)$ is an irreducible polynomial in $k(x)[y]$, it follows that $G$ is a transitive subgroup of $S_4$. It can also be shown that $\phi$ is \'{e}tale away from $\{0,\infty\}$, the inertia groups above $0$ are the cyclic groups generated by the $3$-cycles in $G$; moreover, if $I_\infty \subset G$ occurs as an inertia groups above $\infty$, a $4$-cycle $\tau$ is contained in $I_\infty$. From this it follows that $G = S_4$. Note that $N_{S_4}(\langle \, \tau \, \rangle) \cong \langle \, (1,2,3,4), (1,3) \, \rangle$ is the Dihedral group $D_8$ of order $8$ (which is also a Sylow $2$-subgroup of $S_4$). So $I_\infty = p(I_\infty)$ is either a cyclic group of order $4$ or is a Sylow $2$-group. If $I_\infty \cong \mathbb{Z}/4$ (generated by $\tau$), the induced cover $Z \longrightarrow Y = \mathbb{P}^1_y$ is a connected $S_4$-Galois cover, tamely ramified over $(y=1)$, \'{e}tale everywhere else. By the Riemann Hurwitz formula, such a cover cannot exist. So the cover $\phi$ has the Sylow $2$-group as the inertia groups above $\infty$.
\end{example}

Now we establish the GPWIC (Conjecture~\ref{conj_GPWIC}) in characteristic $2$ for any product of the Alternating and the Symmetric groups we encountered above. Since $S_d$ is not a perfect group, Remark~\ref{rmk_realization_product}~\eqref{item:3} do not apply; instead, we apply patching results by Harbater and Raynaud together with Theorem~\ref{thm_PWIC_A_d_char_2}. The proof is broken down into three steps. After a simple group theoretic observation in Step 1 that the potential purely wild inertia groups contain certain elements $g^{(i)}$'s, we proceed to construct simpler realizable pairs in Step 2 with inertia groups $\langle \, g^{(i)} \, \rangle$. We `patch' together these covers in Step 3 to obtain the result.

\begin{theorem}\label{thm_product_S_d_char_2}
The GPWIC in characteristic $2$ is true for any product $G = S_{d_1} \times \cdots \times S_{d_n}$, $n \geq 1$, where each $d_i \geq 5$ is an odd integer.
\end{theorem}

\begin{proof}
We write any element of $G$ as $g = (g_1, \cdots, g_n)$. For $g \in G$, set $\text{Supp}(g) \coloneqq \{1 \leq i \leq n \, | \, g_i \neq 1\}$. Let $r \geq 1$ be an integer, $B = \{x_1, \ldots, x_r\} \subset \mathbb{P}^1$ be a set of closed points, and $P_1, \ldots, P_r \subset G$ be $2$-groups such that $G = \left\langle \, P_1^G, \ldots, P_r^G \, \right\rangle$ where $P_j^G$ denote the set of conjugates of $P_j$ in $G$. We assert that there is a connected $G$-Galois cover of $\mathbb{P}^1$ that is \'{e}tale away from $B$, and $P_j$ occurs as an inertia group above $x_j$ for $1 \leq j \leq r$. Since the inertia groups above a point in a connected $G$-Galois cover are conjugate, we assume that $P_1, \ldots, P_r$ are contained in a single Sylow $2$-subgroup of $G$.

\underline{Step 1:} We start with a group theoretic observation that each $2$-group $P_j$ contains a product $P_j'$ of cyclic groups, and the conjugates of all these $P_j'$'s generate $G$.

There is a natural projection map $\pi \, \colon \, G \longrightarrow Q \coloneqq \prod_{1 \leq i \leq n} \mathbb{Z}/2$. Since $G$ is generated by the conjugates of $P_j$'s, $Q$ is generated by the conjugates of $\pi(P_j)$'s. As $Q$ is abelian, we have $\left\langle \, \pi(P_1), \ldots, \pi(P_r) \, \right\rangle = Q$. If $\pi(P_{j_0}) = \{1\}$ for some $j_0$, we have $\left\langle \, P_j^G \, | \, 1 \leq j \leq r, \, j \neq j_0 \, \right\rangle = G$. In view of \cite[Theorem~4.7]{Das}, it is enough to consider that $\pi(P_j) \neq \{1\}$ for any $1 \leq j \leq r$. Since each $\pi(P_j) \subset Q$ is a non-trivial elementary abelian $2$-group, there are integers $0 = t_0 < t_1 < \cdots < t_{r-1} < t_r$ such that the following hold.
\begin{enumerate}
\item For each $1 \leq j \leq r$, there is a subgroup $P'_j \coloneqq \left\langle \, g^{(t_{j-1}+1)} \, \right\rangle \times \cdots \times \left\langle \, g^{(t_j)} \, \right\rangle \subset P_j$ such that $\pi(P'_j) = \pi(P_j)$;
\item for each $1 \leq i \leq t_r$, there exists an integer $\lambda(i)$ such that $g^{(i)}_{\lambda(i)} \in S_{\lambda(i)}$ is an odd permutation;
\item $\cup \{ \lambda(i) \, | \, 1 \leq i \leq t_r \} = \{1, \cdots, n\}$, i.e., conjugates of $\left\langle \, P'_j \, | \, 1 \leq j \leq r \, \right\rangle$ generate $G$.
\end{enumerate}

\underline{Step 2:} Let $1 \leq j \leq r$, and $t_{j-1}+1 \leq i \leq t_j$. We use formal patching techniques to realize $\left\langle \, g^{(i)} \, \right\rangle$ as the inertia groups for certain quasi $2$-subgroups of $G$.

First suppose that $g^{(i)}_{\lambda(i)}$ has order $\geq 4$. Set $h^{(i)} \coloneqq (g^{(i)})^2$. Then for any $v \in \text{Supp}\left( h^{(i)} \right)$, \, $h^{(i)}_v \in S_{d_v}$ is a non-trivial even permutation. Set $K_i \coloneqq \prod_{v \in \text{Supp}\left( h^{(i)} \right)} A_{d_v} \subset G$. By Theorem~\ref{thm_PWIC_A_d_char_2}, the pair $\left(K_i, \left\langle \, h^{(i)} \, \right\rangle \right)$ is realizable. As the pair $\left(\left\langle \, g^{(i)} \, \right\rangle, \, \left\langle \, g^{(i)} \, \right\rangle\right)$ is also realizable, by \cite[Theorem~2.2.3]{Raynaud_AC}, we conclude that the pair
\begin{equation}\label{eq_pair_1}
\left(H_i \coloneqq \left\langle \, K_i, \left\langle \, g^{(i)} \, \right\rangle \, \right\rangle, \, \left\langle \, g^{(i)} \, \right\rangle\right) \, \text{ is realizable}.
\end{equation}
Since $g^{(i)}_{\lambda(i)}$ is an odd permutation, the image of $H_i$ under the $\lambda(i)^{\text{th}}$ projection $G \twoheadrightarrow S_{d_{\lambda(i)}}$ is $S_{d_{\lambda(i)}}$.

Now let $\text{ord}\left( g^{(i)}_{\lambda(i)} \right) = 2$. Then $g^{(i)}_{\lambda(i)}$ is conjugate to the element $\tau \coloneqq (1,2) \cdots (2u-1, 2u)$ in $S_{d_{\lambda(i)}}$ for some odd integer $1 \leq u < d/2$, and let $w \in S_{d_{\lambda(i)}}$ be such that $g^{(i)}_{\lambda(i)} = w^{-1} \tau w$. For each $1 \leq t \leq u$ and $2u+1 \leq b \leq d_{\lambda(i)}$, consider the $3$-cycle $\sigma_{t,b} \coloneqq (2t-1,2t,b) \in S_{d_{\lambda(i)}}$, and let $\gamma_{t,b} \in H_i$ be such that $(\gamma_{t,b})_{\lambda(i)} = \sigma_{t,b}$, $(\gamma_{t,b})_{s} = 1$ for $s \neq t$. As in the proof of Proposition~\ref{prop_cyclic_order_two_inertia}, we have $\tau^{-1} \gamma_{t,b} \tau = \gamma_{t,b}^2$. Hence for each $t$ and $b$, setting $\gamma_{t,b}' \coloneqq w^{-1} \gamma_{t,b} w$, we have
$$\left( g^{(i)} \right)^{-1} \gamma_{t,b}' g^{(i)} = \left( \gamma_{t,b}' \right)^2.$$
By \cite[Corollary~2.2.2]{Raynaud_AC}, the pairs $\left( \left\langle \, \gamma_{t,b}' \, \right\rangle \rtimes \left\langle \, g^{(i)} \, \right\rangle, \, \left\langle \, g^{(i)} \, \right\rangle \right)$ are realizable for each $1 \leq t \leq u$ and $2u+1 \leq b \leq d_{\lambda(i)}$. By \cite[Theorem~2.2.3]{Raynaud_AC}, we conclude that the pair
\begin{equation}\label{eq_pair_2}
\left( B_i \coloneqq \left\langle \, \left\langle \, \gamma_{t,b}' \, \right\rangle \rtimes \left\langle \, g^{(i)} \, \right\rangle \, | \, 1 \leq t \leq u, \, 2u+1 \leq b \leq d_{\lambda(i)} \, \right\rangle , \, \left\langle \, g^{(i)} \, \right\rangle \right) \, \text{ is realizable}.
\end{equation}
The image of $B_i$ under the projection $G \twoheadrightarrow S_{d_{\lambda(i)}}$ is $S_{d_{\lambda(i)}}$.

\underline{Step 3:} For $1 \leq j \leq r$, let $G_j$ be the subgroup of $G$ generated by $H_i$'s ($\text{ord}\left( g^{(i)}_{\lambda(i)} \right) \geq 4$) and $B_l$'s ($\text{ord}\left( g^{(l)}_{\lambda(l)} \right) = 2$), $t_{j-1}+1 \leq i, l \leq t_j$. Since $G_j \twoheadrightarrow \prod_{t_{j-1}+1 \leq i \leq t_j} S_{d_{\lambda(i)}}$ by construction, $\left\langle \, G_j \, | \, 1 \leq j \leq r \right\rangle = G$.

Applying \cite[Theorem~2.2.3]{Raynaud_AC} to the realizable pairs in~\eqref{eq_pair_1} and \eqref{eq_pair_2}, for each $1\leq j \leq r$, the pair $(G_j, \, P'_j)$ is realizable; further applying \cite[Theorem~2]{2}, the pair $(G_j, \, P_j)$ is realizable. Finally, we apply \cite[Theorem~4.7]{Das} inductively to obtain our required $G$-Galois cover.
\end{proof}

\begin{corollary}\label{cor_GPWIC_arbit_product_char_2}
The GPWIC (Conjecture~\ref{conj_GPWIC}) is true in characteristic $2$ for any product $G = G_1 \times \cdots \times G_n$, $n \geq 1$, where $G_i = A_{d_i}$ for some $d_i \geq 5$ with $4 \nmid d_i$ or $G_i = S_{d_i}$ for an odd integer $d_i \geq 5$.
\end{corollary}

\begin{proof}
This follows immediately from Corollary~\ref{cor_GPWIC_product_A_d_char_2}, Theorem~\ref{thm_PWIC_S_d_char_2} and \cite[Theorem~7.5]{Das}.
\end{proof}

\bibliographystyle{plainnat}

\end{document}